\documentclass[10pt]{article}%
\usepackage{amsmath}
\usepackage{amsfonts}
\usepackage{mathrsfs}
\usepackage{amssymb}
\usepackage[linkcolor=black,anchorcolor=black,citecolor=black]{hyperref}
\usepackage{graphicx}
\usepackage[body={15cm,20cm}, top=3cm]{geometry}%
\providecommand{\U}[1]{\protect\rule{.1in}{.1in}}
\providecommand{\U}[1]{\protect \rule{.1in}{.1in}}
\newtheorem{theorem}{Theorem}[section]

\newtheorem{corollary}[theorem]{Corollary}

\newtheorem{definition}[theorem]{Definition}
\newtheorem{example}[theorem]{Example}

\newtheorem{lemma}[theorem]{Lemma}

\newtheorem{proposition}[theorem]{Proposition}
\newtheorem{remark}[theorem]{Remark}

\newenvironment{proof}[1][Proof]{\noindent \textbf{#1.} }{\  \rule{0.5em}{0.5em}}
\begin{document}

\title{BSDE, Path-dependent PDE and Nonlinear Feynman-Kac Formula}
\author{Shige Peng\thanks{Partially supported by National Basic Research Program of China (973 Program) (No.
2007CB814906), and NSF of China (No. 10921101).}  and Falei Wang*\\
School of Mathematics, Shandong University\\
250100, Jinan, China}
\date{August 21, 2011}
\maketitle


\begin{abstract}
   In this paper, we introduce a type of path-dependent quasilinear (parabolic) partial differential equations in which the (continuous)
   paths $\omega_t$ on an interval $[0,t]$
   becomes the basic variables  in the place of classical variables
   $(t,x)\in [0,T]\times\mathbb{R}^d$.  This new type of PDE are formulated through a classical backward stochastic differential equations
(BSDEs, for short) in which the terminal values and the generators
are allowed to be general function of Brownian paths.  In this way
we have established a new type of nonlinear Feynman- Kac formula for
a general non-Markovian BSDE. Some main properties of regularities for this new PDE was obtained. 
\end{abstract}

\section{Introduction}

Linear Backward Stochastic Differential Equations (in short BSDE)
was introduced by Bismut \cite{Bismut} in 1973. Pardoux and Peng
\cite[1990]{PP1}  established the existence and uniqueness theorem
for nonlinear BSDEs under a standard Lipschitz condition. Peng
\cite[1991]{Peng1} and then Peng and Pardoux \cite[1992]{PP2}
introduced what we called nonlinear Feynman-Kac formula which
provides a probabilistic representation for a wide class of (system
of) quasi-linear  partial differential equations. Since then,
especially after the paper of El Karoui-Peng-Quenez \cite[1997]{EPQ}
for BSDE, the theory of BSDE have received a wide attention for both
theoretical research and applications.

In order to illustrate the above nonlinear Feynman-Kac formula,  let
us consider
the following well-posed BSDE%
\begin{equation}%
\begin{array}
[c]{rl}%
-dY(t)= & f(t,Y(t),Z(t))dt-Z(t)dB(t),\ \ t\in\lbrack0,T],\\
Y(T)= & \varphi(B(T)),
\end{array}
\tag{BSDE}%
\end{equation}
where $B$ is a given $d$-dimensional Brownian motion, $f=f(t,y,z)$,
$\varphi=\varphi(x)$ are two `good' functions. The unique solution
of this BSDE consists of two stochastic process $(Y(t),Z(t))_{0\leq
t\leq T}$ adapting to the natural filtration
$(\mathcal{F}_{t})_{t\geq0}$ of Brownian motion $B$ (see Section
2.2. for details), namely, both $Y(t)$ and $Z(t)$ at each time
$t\in\lbrack0,T]$ are functions of the Brownian path $B(s)_{0\leq
s\leq t}$. For the case $Y(T)=\varphi(B(T))$, there exists a
deterministic function $u=u(t,x)$ defined on
$[0,T]\times\mathbb{R}^{d}$ such that $Y(t)=u(t,B(t))$ and
$Z(t)=\nabla u(t,B(t))$. We can also prove that the function
$u=u(t,x)$, $t\in\lbrack0,T)$, $x\in\mathbb{R}^{d}$, is in fact the
unique solution of the following quasi-linear PDE:
\begin{equation}
\partial_{t}u+\frac{1}{2}\Delta u+f(t,u,\nabla u)=0,\ \ u(T,x)=\varphi
(x).\tag{PDE}%
\end{equation}
This relation permits us to solve the above type of BSDE by PDE.
Conversely we can also use the BSDE to solve the PDE. This nonlinear
Feynman-Kac formula also provides a nonlinear Monte-Carlo method via
the BSDE (\cite{Bouchard-Touzi}, \cite{GEJW}) to solve numerically the PDE.

But in general the terminal condition $Y(T)$ of the BSDE may be a
general function of Brownian paths, i.e., $Y(T)=\varphi(B(s)_{0\leq
t\leq T})$, where $\varphi$ is a given functional defined on the
space $C([0,T],\mathbb{R}^{d})$ of $d$-dimensional continuous paths.
In this situation the solution $(Y,Z)$ is regarded as a generalized
`path-dependent' solution of the above BSDE. This problem was raised
by the first author and summarized in his ICM2010's talk
\cite[Sec.1.3]{PengICM2010}.

With a very different point of view, Dupire \cite[2009]{Dupire}
introduced a new\ type of a functional It\^{o}'s formula which
non-trivially generalized the classical one (See Cont and
Fourni\'{e} \cite{CF1, CF2, CF3} for a more general and more
systematic research). His idea is to introduce a simple but deeply
insightful definition of derivatives with respect to path for a
given $(\mathcal{F}_{t})_{t\geq0}$-adapted process
$Y(t)=u(t,\omega(s)_{0\leq s\leq t})$, which is a family of
functional of Brownian path indexed by time $t$. He then claimed
that if a martingale is a $C^{1,2}$-function in his new definition
then it is also a solution of a `functional PDE'. The so-called
functional Feynman-Kac formula was also claimed under the same
framework.

In this paper we will prove that, under certain smooth assumptions
of the given path-dependent  functions
\[
Y(T)=\Phi(\omega(s)_{0\leq s\leq T}),\ \ \ f=f(t,\omega(s)_{0\leq
s\leq t},y,z)
\]
with respect to $(y,z)$ and the path $\omega$ (see$(\mathbf{ H1})$
and $(\mathbf{H2})$ in Section 3), the solution $(Y(t),Z(t))$ of
(BSDE) solves (PDE). More specifically, the path-function
$u(t,\omega(s)_{0\leq s\leq t}):=Y(t,\omega)$ is the unique
$C^{1,2}$-solution of (PDE) and $Z(t)$ is the vertical derivative.
We then can prove the regularity of the solution of BSDE in the
sense of Dupire's derivatives which gives the existence of $Y(t)$.
The results of this paper  has non-trivially generalized the ones of
Pardoux and Peng \cite[1992]{PP1} to the path-dependent situation.

The paper is organized as follows. In section 2, we present some
existing results on functional It\^{o}'s formula and BSDE that will
be used in this paper. In section 3, we establish some estimates and
regularity results for the solution of BSDEs with respect to paths.
Finally, in section 4, we establish our main results, Theorems 4.1
and 4.3 which provide a one to one correspondence between (BSDE) and
the system of path-dependent PDE.

When the coefficients of (BSDE) are only Lipschitz functions, we
usually can not obtain the smoothness results given in this paper,
thus a new type of viscosity solution  is required. We refer to Peng
\cite[2011]{Peng3} for the corresponding comparison theorem.

\section{Preliminaries}

\subsection{Functional It\^{o}'s formula}

The following notations are mainly from Dupire \cite{Dupire} . Let $T>0$ be fixed. For each $t\in\lbrack0,T]$, we denote by
$\Lambda_{t}$ the set of c\`{a}dl\`{a}g (French abbreviation for
\textquotedblleft right continuous with left limit\textquotedblright, also
often denoted by RCLL) $\mathbb{R}^{d}$-valued functions on $[0,t]$.

For each $\gamma\in\Lambda_{T}$ the value of $\gamma$ at time $s\in
\lbrack0,T]$ is denoted by $\gamma(s)$. Thus $\gamma=\gamma(s)_{0\leq s\leq
T}$ is a c\`{a}dl\`{a}g process on $[0,T]$ and its value at time $s$ is
$\gamma(s)$. The path of $\gamma$ up to time $t$ is denoted by $\gamma_{t}$,
i.e., $\gamma_{t}=\gamma(s)_{0\leq s\leq t}\in\Lambda_{t}$. We denote
$\Lambda=\bigcup_{t\in\lbrack0,T]}\Lambda_{t}$. We sometimes also specifically
write
\[
\gamma_{t}=\gamma(s)_{0\leq s\leq t}=(\gamma(s)_{0\leq s<t},\gamma(t))
\]
to indicate the terminal position $\gamma(t)$ of $\gamma_{t}$ which often
plays a special role in this framework. For each $\gamma_{t}\in\Lambda$ and
$x\in\mathbb{R}^{d}$ we denote $\gamma_{t}^{x}=(\gamma(s)_{0\leq s<t}%
,\gamma(t)+x)$ which is also an element in $\Lambda_{t}$.

We are interested in a function $u$ of path, i.e., $u:\Lambda\mapsto
\mathbb{R}$. This function $u=u(\gamma_{t})$, $\gamma_{t}\in\Lambda$ can be
also regarded as a family of real valued functions:
\[
u(\gamma_{t})=u(t,\gamma(s)_{0\leq s\leq t})=u(t,\gamma(s)_{0\leq s<t}%
,\gamma(t)):\gamma_{t}\in\Lambda_{t},\ \ t\in\lbrack0,T].
\]
We also denote $u(\gamma_{t}^{x}):=u(t,\gamma(s)_{0\leq s<t},\gamma(t)+x)$,
for $\gamma_{t}\in\Lambda_{t}$, $x\in\mathbb{R}^{d}$.

\begin{remark}
It is also very important to understand $u(\gamma_{t}^{x})$ as a function
of\ $t$, $(\gamma(s)_{0\leq s<t},\gamma(t))$ and $x$. A typical case is
$u(\gamma_{t})=u(t,\gamma(t_{-}),\gamma(t)+x)$, $t\in\lbrack0,T]$, where
$\gamma(t_{-})=\lim_{s\uparrow t}\gamma(s)$.
\end{remark}

We now introduce a distance on $\Lambda$. Let $\left\langle \cdot\right\rangle
$ and $|\cdot|$ denote the inner product and norm in $\mathbb{R}^{d}$. For
each $0\leq t\leq\bar{t}\leq T$ and $\gamma_{t},\bar{\gamma}_{\bar{t}}%
\in\Lambda$, we denote
\begin{align*}
&  \|\gamma_{t}\|:=\sup\limits_{r\in\lbrack0,t]}|\gamma(r)|,\\
&  d_{\infty}(\gamma_{t},\bar{\gamma}_{\bar{t}}):=\max(\sup\limits_{r\in
\lbrack0,t)}\{|\gamma(r)-\bar{\gamma}(r)|\},\sup_{r\in\lbrack t,\bar{t}%
]}\{|\gamma(t)-\bar{\gamma}(r)|\})+|t-\bar{t}|.
\end{align*}

It is obvious that $\Lambda_{t}$ is a Banach space with respect to $\Vert
\cdot\Vert$. Since $\Lambda$ is not a linear space, $d_{\infty}$ is not a norm.

\begin{definition}
(Continuous) A functionals $u:\Lambda\mapsto\mathbb{R}$ is said $\Lambda
$-continuous at $\gamma_{t}\in\Lambda$, if for any $\varepsilon>0$ there
exists $\delta>0$ such that for each $\bar{\gamma}_{\bar{t}}\in\Lambda$
satisfying $d_{\infty}(\gamma_{t},\bar{\gamma}_{\bar{t}})<\delta$, we have
$|u(\gamma_{t})-u(\bar{\gamma}_{\bar{t}})|<\varepsilon$. $u$ is said to be
$\Lambda$-continuous if it is $\Lambda$-continuous at each $\gamma_{t}%
\in\Lambda$.
\end{definition}

\begin{remark}
In our framework we often regard $u(\gamma_{t}^{x})$ as a function of $t$,
$\gamma$ and $x$, i.e., $u(\gamma_{t}^{x})=u(t,\gamma(s)_{0\leq s<t}%
,\gamma(t)+x)$. Thus, for a fixed $\gamma\in\Lambda_{T}$, $u(\gamma_{t}^{x})$
is regarded as a function of $(t,x)\in\lbrack0,T]\times\mathbb{R}^{d}$.
\end{remark}

\begin{definition}
Let $u:\Lambda\mapsto\mathbb{R}$ and $\gamma_{t}\in\Lambda$ be given. If there
exists $p\in\mathbb{R}^{d}$, such that
\[
u(\gamma_{t}^{x})=u(\gamma_{t})+\left\langle p,x\right\rangle +o(|x|),\ x\in
\mathbb{R}^{d},\ \
\]
Then we say that $u$ is (vertically) differentiable at $\gamma_{t}$ and denote
the gradient of $D_{x}u(\gamma_{t})=p$. $u$ is said to be vertically
differentiable in $\Lambda$ if $D_{x}u(\gamma_{t})$ exists for each
$\gamma_{t}\in\Lambda$. We can similarly define the Hessian $D_{xx}%
u(\gamma_{t})$. It is an $\mathbb{S}(d)$-valued function defined on $\Lambda$,
where $\mathbb{S}(d)$ is the space of all $d\times d$ symmetric matrices.
\end{definition}

For each $\gamma_{t}\in\Lambda$ we denote
\[
\gamma_{t,s}(r)=\gamma(r)\mathbf{1}_{[0,t)}(r)+\gamma(t)\mathbf{1}%
_{[t,s]}(r),\ \ r\in\lbrack0,s].
\]
It is clear that $\gamma_{t,s}\in\Lambda_{s}$.

\begin{definition}
For a given $\gamma_{t}\in\Lambda$ if we have
\[
u(\gamma_{t,s})=u(\gamma_{t})+a(s-t)+o(|s-t|),\ \ s\geq t,\ \
\]
then we say that $u(\gamma_{t})$ is (horizontally) differentiable in $t$ at
$\gamma_{t}$ and denote $D_{t}u(\gamma_{t})=a$. $u$ is said to be horizontally
differentiable in $\Lambda$ if $D_{t}u(\gamma_{t})$ exists for each
$\gamma_{t}\in\Lambda$.
\end{definition}

\begin{definition}
Define $\mathbb{C}^{j,k}(\Lambda)$ as the set of functionals $u:=(u(\gamma
_{t}))_{\gamma_{t}\in\Lambda}$ defined on $\Lambda$ which are $j$ times
horizontally and and $k$ times vertically differentiable in $\Lambda$ such
that all these derivatives are $\Lambda$-continuous.
\end{definition}

\begin{definition}
$u:=(u(\gamma_{t}))_{\gamma_{t}\in\Lambda}$ is said to be in $\mathbb{C}%
_{l,lip}^{0,2}(\Lambda)$, if $u$ is in $\mathbb{C}^{0,2}(\Lambda)$
and, for $\varphi=u,D_{x}u$, $D_{xx}u$, we have
\[
|\varphi(\gamma_{t})-\varphi(\bar{\gamma}_{\bar{t}})|\leq C(1+\|\gamma
_{t}\|^{q}+\|\bar{\gamma}_{\bar{t}}\|^{q})d_{\infty}(\gamma_{t},\bar{\gamma
}_{\bar{t}}),\ \ \ \gamma_{t},\bar{\gamma}_{t}\in\Lambda,
\]
where $C,q$ are constants depending only on $\varphi$.
\end{definition}

\begin{example}
\label{w12} If $u(\gamma_{t})=f(t,\gamma_{t}(t))$ with $f\in C^{1,1}%
([0,T[\times\mathbb{R})$, we have
\[
D_{t}F(\gamma_{t})=\partial_{t}f(t,\gamma_{t}(t)),\ \ \ D_{x}F_{t}(\gamma
_{t})=\partial_{x}f(t,\gamma_{t}(t)),
\]
which is the classic derivative. In general, these derivatives also satisfy
the classic properties: linearity, product and chain rule.
\end{example}

The following It\^{o} formula was firstly obtained by Dupire \cite[2009]{Dupire} and then by Cont and Fourni\'{e} \cite[2010]{CF2}  for a more general formulation.

\begin{theorem}
(Functional It\^{o}'s formula).\label{w2} Let $(\Omega,\mathcal{F}%
,(\mathcal{F}_{t})_{t\in\lbrack0,T]},P)$ be a probability space, if $X$ is a
continuous semi-martingale and $u$ is in $\mathbb{C}^{1,2}(\Lambda)$, then for
any $t\in\lbrack0,T[$:
\begin{align*}
u(X_{t})-u(X_{0})=  &  \int_{0}^{t}D_{s}u(X_{s})ds+\int_{0}^{t}D_{x}%
u(X_{s})dX(s)+\frac{1}{2}\int_{0}^{t}D_{xx}u(X_{s})d\langle X\rangle(s),\ \ \ a.s..
\end{align*}
In particular , if $X$ is a  $\mathcal{F}$-Brownian motion,
\begin{align*}
u(X_{t})-u(X_{0})=  &  \int_{0}^{t}D_{s}u(X_{s})ds+\int_{0}^{t}D_{x}%
u(X_{s})dX(s)+\frac{1}{2}\int_{0}^{t}D_{xx}u(X_{s})ds,\ \ \ a.s..
\end{align*}
\end{theorem}

\subsection{Backward Stochastic Differential Equations}

Let $\Omega=C([0,T],\mathbb{R}^{d})$ and $P$ the standard Wiener measure
defined on $(\Omega,\mathcal{B}(\Omega))$. We denote by the canonical process
$B(t)=B(t,\omega)=\omega(t)$, $t\in\lbrack0,T]$, $\omega\in\Omega$. Then
$B(t)_{0\leq t\leq T}$ is a $d$-dimensional Brownian motion defined on the
probability space $(\Omega,\mathcal{B}(\Omega),P)$. Let $\mathcal{N}$ be the
collection of all $P$-null sets in $\Omega$. For any $0\leq t\leq r\leq T$,
$\mathcal{F}_{r}^{t}$ denotes the completion of $\sigma(B(s)-B(t);t\leq s\leq
r)$, i.e., $\mathcal{F}_{r}^{t}=\sigma\{\sigma(B(s)-B(t);t\leq s\leq
r)\vee\mathcal{N}\}$. We also write $\mathcal{F}_{r}$ for $\mathcal{F}_{r}%
^{0}$ and $\mathcal{F}^{t}$ for $\mathcal{F}_{T}^{t}$.

For any $0\leq t\leq T$ , we denote by $L^{2}(\mathcal{F}_{t})$ the
set of all square integrable $\mathcal{F}_{t}$-measurable random
variables, $M^{2}(t,T)$ the space of all
$\mathcal{F}_{s}^{t}$-adapted, $\mathbb{R}^{d}$-valued
processes $(X(s))_{s\in\lbrack t,T]}$ with $E[\int_{t}^{T}|X(s)|^{2}%
ds]<\infty$ and $S^{2}(t,T)$ the space of all $\mathcal{F}_{s}^{t}$-adapted,
$\mathbb{R}^{d}$-valued continuous processes $(X(s))_{s\in\lbrack t,T]}$ with
$E[\sup\limits_{s\in\lbrack t,T]}|X(s)|^{2}]<\infty$.

Let us consider a deterministic function $f:\Lambda\times\mathbb{R}^{m}%
\times\mathbb{R}^{m\times d}\mapsto\mathbb{R}^{m}$, which will be in the
following the generator of our BSDEs. For the function $f$, we will make the
following assumptions:\newline$\bullet$ $f(\gamma_{t},y,z)$ is a given
continuous function on $\Lambda\times\mathbb{R}^{m}\times\mathbb{R}^{m\times
d}$. \newline$\bullet$ There exists constants $C\geq0$ and $q\geq0$ such that:
for $\gamma_{t},\bar{\gamma}_{t}\in\Lambda,\ y,\bar{y}\in\mathbb{R}^{m}%
,z,\bar{z}\in\mathbb{R}^{m\times d}$,
\[
|f(\gamma_{t},y,z)-f(\bar{\gamma}_{t},\bar{y},\bar{z})|\leq C((1+\|\gamma
_{t}\|^{q}+\|\bar{\gamma}_{t}\|^{q})\|\gamma_{t}-\bar{\gamma}_{t}%
\|+|y-\bar{y}|+|z-\bar{z}|).
\]

The following result on backward stochastic differential equations
(BSDEs) is by now well known, for its proof the reader is referred
to Pardoux and Peng \cite{PP1}.
\begin{lemma}
Let $f$ satisfy the above conditions, then for each $\xi\in L^{2}%
(\mathcal{F}_{T})$, the BSDE
\begin{align}\label{BSDE1}
Y(t)=\xi+\int_{t}^{T}f(B_{s},Y(s),Z(s))ds-\int_{t}^{T}Z(s)dB(s),\ \ 0\leq
t\leq T,
\end{align}
has a unique adapted solution \[(Y(t),Z(t))_{0\leq t\leq T}\in
S^{2}(0,T)\times M^{2}(0,T).\]
\end{lemma}

We also shall recall the following  basic result on BSDEs, which is
the well-known comparison theorem (see El Karoui, Peng, Quenez
\cite{EPQ}).
\begin{lemma}\label{comp}
(Comparison Theorem) We assume $m = 1$. Given two coefficients $f_1$
and $f_2$ satisfying the above assumptions and two terminal values
$\xi_1$, $\xi_2 \in L^{2}(\mathcal{F}_{t})$ , we denote by $(Y_1,
Z_1)$ and $(Y_2, Z_2)$ the solution of BSDE with the data $(\xi_1,
f_1)$ and $(\xi_2, f_2)$, respectively. Then we have:  If $\xi_1
\geq \xi_2$ and $f_1 \geq f_2 $, a.s., then $Y_1(t)\geq Y_2(t)$ ,
a.s., for all $t \in [0, T]$.
\end{lemma}

\section{Property of solution of the BSDE}

Let us first recall some notations from Pardoux and Peng \cite{Peng1}.

$C^{k}(\mathbb{R}^{p},\mathbb{R}^{q}),C^{k}_{b}(\mathbb{R}^{p},\mathbb{R}%
^{q}), C_{p}^{k}(\mathbb{R}^{p},\mathbb{R}^{q})$ will denote respectively the
set of functions of class $C^{k}$ from $\mathbb{R}^{p}$ into $\mathbb{R}^{q}$,
the set of functions the set of those functions of class $C^{k}$ whose partial
derivatives of order less than or equal to $k$ are bounded (and hence the
function itself grows at most linearly at infinity), and the set of those
functions of class $C^{k}$ which, together with all their partial derivatives
of order less than or equal to $k$, grow at most like a polynomial function of
the variable $x$ at infinity.

The following directional derivatives will be used frequently in the sequel.

\begin{definition}
If $\Phi$ is an $\mathbb{R}^{m}$-valued function on $\Lambda_{T}$.
$\Phi$ is said to be in $C^{2}(\Lambda_{T})$, if for each
$\gamma\in\Lambda_{T}$ and $t\in\lbrack0,T]$, there exist
$p_{1}\in\mathbb{R}^{d}$ and $p_{2}\in
\mathbb{R}^{d}\otimes\mathbb{R}^{d}$  such that $p_{2}$ is symmetric
and
\[
\Phi(\gamma_{\gamma_{t}^{x}})-\Phi(\gamma)=\left\langle p_{1},x\right\rangle
+\frac{1}{2}\left\langle p_{2},x\otimes x\right\rangle  +o(|x|^{2}%
),\ \ x\in\mathbb{R}^{d},
\]
where $\gamma_{\gamma_{t}^{x}}(r)=\gamma(r)I_{[0,t)}(r)+(\gamma(r)+x)I_{[t,T]}%
(r)$. We denote
\[
\Phi_{\gamma_{t}}^{\prime}(\gamma):=p_{1}\text{, }\Phi_{\gamma_{t}}%
^{\prime\prime}(\gamma):=p_{2}.
\]
$\Phi$ is said to be in $C_{l,lip}^{2}(\Lambda_{T})$ if $\Phi_{\gamma_{t}%
}^{\prime}(\gamma)$ and $\Phi_{\gamma_{t}}^{\prime\prime}(\gamma)$
exist for all $\gamma \in\Lambda_{T}, t\in [0,T]$ , and for some
constants $C,k\geq0$ depending only on $\Phi$,
\begin{align*}
&|\Phi(\gamma)-\Phi(\bar{\gamma})|\leq
C(1+\Vert\gamma\Vert^{k}+\Vert
\bar{\gamma}\Vert^{k})\Vert\gamma-\bar{\gamma}\Vert,\ \ \
\gamma,\bar{\gamma }\in\Lambda_{T},\\&
|\Psi_{\gamma_{t}}(\gamma)-\Psi_{\bar{\gamma}_{s}}(\bar{\gamma})|\leq
C(1+\Vert\gamma\Vert^{k}+\Vert\bar{\gamma}\Vert^{k})(|t-s|+\Vert\gamma
-\bar{\gamma}\Vert),\ \ \
\gamma,\bar{\gamma}\in\Lambda_{T},t,s\in\lbrack0,T],
\end{align*}
with $\Psi_{\gamma_{t}}(\gamma)=\Phi_{\gamma_{t}}^{\prime}(\gamma
),\Phi_{\gamma_{t}}^{\prime\prime}(\gamma)$. Analogous, we can
define $C^{2}(\Lambda_{t})$, $C_{l,lip}^{2}(\Lambda_{t})$,
$C_{l,lip}^{1}(\Lambda_{t})$.
\end{definition}

In this rest of this paper we shall make use of the following assumptions on
the generator $f$ and the terminal $\Phi$ of our BSDE:

(\textbf{H1}): Suppose $\Phi$ is an $\mathbb{R}^{m}$-valued function
on $\Lambda_{T}$.\ Moreover, $\Phi$ is in
$C_{l,lip}^{2}(\Lambda_{T})$.

(\textbf{H2}): Suppose $f(\gamma_{t},y,z)$ is a given continuous
function on $\Lambda\times\mathbb{R}^{m}\times\mathbb{R}^{m\times
d}$.  For any $\gamma_{t}\in\Lambda$ and $s\in\lbrack0,t]$,
$(x,y,z)\mapsto f((\gamma_{t})_{\gamma_{s}^{x}},y,z)$ is of class $C_{p}%
^{3}(\mathbb{R}^{d}\times\mathbb{R}^{m}\times\mathbb{R}^{m\times d}%
,\mathbb{R}^{m})$ and the first order partial derivatives as well as
their derivatives of up to order two with respect to $(y,z)$ are
uniformly bounded; for any
$(y,z)$, $\gamma_{t}\mapsto f(\gamma_{t},y,z)$ is of class $C_{l,lip}%
^{2}(\Lambda_{t})$, $\gamma_{t}\mapsto f_y(\gamma_{t},y,z), f_z(\gamma_{t},y,z) $ is of class $C_{l,lip}%
^{1}(\Lambda_{t})$, $\gamma_{t}\mapsto f_{yy}(\gamma_{t},y,z),
f_{zz}(\gamma_{t},y,z), f_{yz}(\gamma_{t},y,z) $ is of class
$C_{l,lip}(\Lambda_{t})$.

 (\textbf{H3}): Suppose
$f(\gamma_{t},y,z)=\bar{f}(t,\gamma_t(t),y,z)$, where
$\bar{f}:[0,T]\times\mathbb{R}^{d}\times\mathbb{R}^{m}\times
\mathbb{R}^{m\times d}\mapsto\mathbb{R}^{m}$ is such that
$(t,x,y,z)\mapsto \bar{f}(t,x,y,z)$ is of class
$C_{p}^{0,3}([0,T]\times\mathbb{R}^{d}\times
\mathbb{R}^{m}\times\mathbb{R}^{m\times d},\mathbb{R}^{m})$ and the
first order partial derivatives as well as their derivatives of up
to order two with respect to $(y,z)$ are uniformly bounded.

It is obvious that assumption $(\mathbf{H}3)$ implies assumption
$(\mathbf{H}2)$.

Assume $(\mathbf{H}1)$ and $(\mathbf{H}2)$ hold. For any $\gamma_{t}%
=(\gamma(s))_{0\leq s\leq t}\in\Lambda$, let $(Y_{{\gamma_{t}}}(s),Z_{{\gamma
_{t}}}(s))_{t\leq s\leq T}$ denote the unique element of $S^{2}[t,T]\times
M^{2}[t,T]$ which solves the following BSDE:
\begin{equation}
Y_{{\gamma_{t}}}(s)=\Phi(B^{\gamma_{t}})+\int_{s}^{T}f(B_{r}^{\gamma_{t}%
},Y_{{\gamma_{t}}}(r),Z_{{\gamma_{t}}}(r))dr-\int_{s}^{T}Z_{{\gamma_{t}}%
}(r)dB(r), \label{p1}%
\end{equation}
where,%
\[
B^{\gamma_{t}}(u):=\gamma_{t}(u)I_{[0,t]}(u)+(\gamma_{t}%
(t)+B(u)-B(t))I_{(t,T]}(u).
\]
It is clear, for $\gamma_{t},\bar{\gamma}_{\bar{t}}\in\Lambda$ with $\bar
{t}\geq t$, we have
\begin{align*}
B^{\gamma_{t}}(u)-B^{\bar{\gamma}_{\bar{t}}}(u)  &  =I_{[0,t)}(u)(\gamma
_{t}(u)-\bar{\gamma}_{\bar{t}}(u))+I_{[t,\bar{t})}(u)(\gamma_{t}%
(t)-\bar{\gamma}_{\bar{t}}(u)+B(u)-B(t))\\
&  +I_{[\bar{t},T]}(u)(\gamma_{t}(t)+B(\bar{t})-B(t)-\bar{\gamma}(\bar{t})).
\end{align*}

It follows easily from the existence result in \cite{PP2}:

\begin{corollary}
For each $\gamma_{t}\in\Lambda$, the BSDE (\ref{p1}) has a unique solution
$(Y_{{\gamma_{t}}}(s),Z_{{\gamma_{t}}}(s))_{t\leq s\leq T}$ and $Y_{{\gamma
_{t}}}(t)$ defines a deterministic mapping from ${\Lambda}$ to $\mathbb{R}%
^{m}$.
\end{corollary}

\subsection{ Regularity of solution of the BSDE}

We next establish higher order moment estimates for the solution of BSDE
(\ref{p1}). For each $z\in\mathbb{R}^{m\times d}$, we denote $\Vert
z\Vert=\sqrt{tr(zz^{\ast})}$.

\begin{lemma}
For any $p\geq2$, there exist $C_{p}$ and $q$ depending on $C,T,k,p$, such
that:
\begin{align}
E[\sup\limits_{s\in\lbrack t,T]}|Y_{{\gamma_{t}}}(s)|^{p}]  &  \leq
C_{p}(1+\Vert\gamma_{t}\Vert^{q}),\\
E[|\int_{t}^{T}\Vert Z_{{\gamma_{t}}}(s)\Vert^{2}ds|^{\frac{p}{2}}]  &  \leq
C_{p}(1+\Vert\gamma_{t}\Vert^{q}).
\end{align}
\label{w4}
\end{lemma}

\begin{proof}
Applying It\^{o}'s formula to $\varphi(x)=|x|^{p}$ yields that,%

\begin{align*}
&  |Y_{{\gamma_{t}}}(s)|^{p}+\frac{p}{2}\int_{s}^{T}|Y_{{\gamma_{t}}%
}(r)|^{p-2}\Vert Z_{{\gamma_{t}}}(r)\Vert^{2}dr+\frac{p}{2}(p-2)\int_{s}%
^{T}|Y_{\gamma_{t}}(r)|^{p-4}\langle Z_{\gamma_{t}}Z_{\gamma_{t}}^{\ast
}Y_{\gamma_{t}},Y_{\gamma_{t}}\rangle(r)dr\\
=  &  |\Phi(B^{\gamma_{t}})|^{p}+p\int_{s}^{T}|Y_{{\gamma_{t}}}%
(r)|^{p-2}\langle Y_{{\gamma_{t}}}(r),f(B_{r}^{\gamma_{t}},Y_{{\gamma_{t}}%
}(r),Z_{{\gamma_{t}}}(r))\rangle dr-p\int_{s}^{T}|Y_{{\gamma_{t}}}%
(r)|^{p-2}\langle Y_{{\gamma_{t}}}(r),Z_{{\gamma_{t}}}(r)dB(r)\rangle.\
\end{align*}
After localization, we get that
\begin{align*}
&  E[|Y_{{\gamma_{t}}}(s)|^{p}+\frac{p}{2}(p-1)\int_{s}^{T}|Y_{{\gamma_{t}}%
}(r)|^{p-2}\Vert Z_{{\gamma_{t}}}(r)\Vert^{2}dr]\\
\leq &  E[|\Phi(B^{\gamma_{t}})|^{p}+p\int_{s}^{T}|Y_{{\gamma_{t}}}%
(r)|^{p-2}\langle Y_{{\gamma_{t}}}(r),f(B_{r}^{\gamma_{t}},Y_{{\gamma_{t}}%
}(r),Z_{{\gamma_{t}}}(r))\rangle dr].
\end{align*}
From H\"older's inequality and Young's inequality, we have for any
$s\leq T$ and any $\delta>0$ that,
\begin{align*}
&  \int_{s}^{T}|Y_{{\gamma_{t}}}(r)|^{p-2}\langle Y_{{\gamma_{t}}}%
(r),f(B_{r}^{\gamma_{t}},Y_{{\gamma_{t}}}(r),Z_{{\gamma_{t}}}(r))\rangle dr\\
\leq &  \int_{s}^{T}|Y_{{\gamma_{t}}}(r)|^{p-1}|f(B_{r}^{\gamma_{t}%
},0,0)|dr+C\int_{s}^{T}|Y_{{\gamma_{t}}}(r)|^{p-2}|Y_{{\gamma_{t}}}(r)|\Vert
Z_{{\gamma_{t}}}(r)\Vert dr\\
\leq &  (\frac{p-1}{p}+\frac{C}{\delta})\int_{s}^{T}|Y_{{\gamma_{t}}}%
(r)|^{p}dr+\frac{1}{p}\int_{s}^{T}|f(B_{r}^{\gamma_{t}},0,0)|^{p}dr+\delta
C\int_{s}^{T}|Y_{{\gamma_{t}}}(r)|^{p-2}\Vert Z_{{\gamma_{t}}}(r)\Vert^{2}dr.
\end{align*}
Hence,
\begin{align*}
&  E[|Y_{{\gamma_{t}}}(s)|^{p}+(\frac{p}{2}(p-1)-\delta Cp)\int_{s}%
^{T}|Y_{{\gamma_{t}}}(r)|^{p-2}\Vert Z_{{\gamma_{t}}}(r)\Vert^{2}dr]\\
\leq &  E[|\Phi(B^{\gamma_{t}})|^{p}+\frac{1}{p}\int_{s}^{T}|f(B_{r}%
^{\gamma_{t}},0,0)|^{p}dr+(p-1+\frac{Cp}{\delta})\int_{s}^{T}|Y_{{\gamma_{t}}%
}(r)|^{p}dr].
\end{align*}
Choosing $\delta$ small enough such that $\frac{p}{2}(p-1)-\delta Cp>0$ and
applying Gronwall's inequality implies that
\[
\sup\limits_{s\in\lbrack t,T]}E[|Y_{{\gamma_{t}}}(s)|^{p}+\int_{s}%
^{T}|Y_{{\gamma_{t}}}(r)|^{p-2}\Vert Z_{{\gamma_{t}}}(r)\Vert^{2}dr]\leq
C_{p}^{1}(1+\Vert\gamma_{t}\Vert^{q}).
\]
On the other hand, still using the first equality of this proof and choosing
$\delta$ appropriately, we deduce the existence of a constant $C_{p}^{2}$ such
that
\begin{align*}
|Y_{{\gamma_{t}}}(s)|^{p}\leq &  |\Phi(B^{\gamma_{t}})|^{p}+C_{p}^{2}\int%
_{s}^{T}(|Y_{{\gamma_{t}}}(r)|^{p}+|f(B_{r}^{\gamma_{t}},0,0)|^{p})dr\\
&  -p\int_{s}^{T}|Y_{{\gamma_{t}}}(r)|^{p-2}\langle Y_{{\gamma_{t}}%
}(r),Z_{{\gamma_{t}}}(r)dB(r)\rangle.
\end{align*}
Hence, from Burkholder-Davis-Gundy's inequality,
\begin{align*}
E[\sup\limits_{s\in\lbrack t,T]}|Y_{{\gamma_{t}}}(s)|^{p}]\leq &
E[|\Phi(B^{\gamma_{t}})|^{p}+C_{p}^{2}\int_{t}^{T}|Y_{{\gamma_{t}}}%
(r)|^{p}+|f(B_{r}^{\gamma_{t}},0,0)|^{p}dr]\\
&  +C_{p}^{\prime}\sqrt{E[\int_{t}^{T}|Y_{{\gamma_{t}}}(r)|^{2p-2}\Vert
Z_{{\gamma_{t}}}(r)\Vert^{2}dr]}.
\end{align*}
Then we can get
\[
E[\sup\limits_{s\in\lbrack t,T]}|Y_{{\gamma_{t}}}(s)|^{p}]\leq C_{p}%
(1+\Vert\gamma_{t}\Vert^{q}).
\]
Now according to It\^{o}'s formula, we have
\begin{align*}
\int_{t}^{T}\Vert Z_{{\gamma_{t}}}(r)\Vert^{2}dr=  &  |\Phi(B^{\gamma_{t}%
})|^{2}-|Y_{{\gamma_{t}}}(t)|^{2}+2\int_{t}^{T}\langle Y_{{\gamma_{t}}%
}(r),f(B_{r}^{\gamma_{t}},Y_{{\gamma_{t}}}(r),Z_{{\gamma_{t}}}(r))\rangle dr\\
&  -2\int_{s}^{T}\langle Y_{{\gamma_{t}}}(r),Z_{{\gamma_{t}}}(r)dB(r)\rangle,
\end{align*}
thus
\begin{align*}
E[|\int_{t}^{T}\Vert Z_{{\gamma_{t}}}(r)\Vert^{2}dr|^{\frac{p}{2}}]\leq &
C_{p}^{3}E[|\Phi(B^{\gamma_{t}})|^{p}+|Y_{{\gamma_{t}}}(t)|^{p}+|\int_{t}%
^{T}\langle Y_{{\gamma_{t}}}(r),f(B_{r}^{\gamma_{t}},Y_{{\gamma_{t}}%
}(r),Z_{{\gamma_{t}}}(r))\rangle dr|^{\frac{p}{2}}\\
&  +|\int_{t}^{T}(Y_{{\gamma_{t}}}(r),Z_{{\gamma_{t}}}(r)dB(r))|^{\frac{p}{2}%
}].
\end{align*}
From H\"older's inequality and Young's inequality
\begin{align*}
&  E[|\int_{t}^{T}\langle Y_{{\gamma_{t}}}(r),f(B_{r}^{\gamma_{t}}%
,Y_{{\gamma_{t}}}(r),Z_{{\gamma_{t}}}(r))\rangle dr|^{\frac{p}{2}}]\\
\leq &  C_{p}^{4}E[|\int_{t}^{T}|f(B_{r}^{\gamma_{t}},0,0)|^{p}dr+\sup
\limits_{s\in\lbrack t,T]}|Y_{{\gamma_{t}}}(s)|^{p}+(\int_{t}^{T}%
|Y_{{\gamma_{t}}}(s)|\Vert Z_{{\gamma_{t}}}(s)\Vert ds)^{\frac{p}{2}}].
\end{align*}
For any $\delta>0$,
\begin{align*}
E[\int_{t}^{T}  &  \Vert Z_{{\gamma_{t}}}(r)\Vert^{2}dr]\leq C_{p}^{4}%
E[|\int_{t}^{T}|f(B_{r}^{\gamma_{t}},0,0)|^{p}dr+(1+\frac{1}{2\delta}%
)\sup\limits_{s\in\lbrack t,T]}|Y_{{\gamma_{t}}}(s)|^{p}\\
+  &  \delta(\int_{t}^{T}\Vert Z_{{\gamma_{t}}}(s)\Vert^{2}ds)^{\frac{p}{2}%
}+|\Phi(B^{\gamma_{t}})|^{p}+|Y_{{\gamma_{t}}}(t)|^{p}+(\int_{t}%
^{T}|Y_{{\gamma_{t}}}(s)|^{2}\Vert Z_{{\gamma_{t}}}(s)\Vert^{2}ds)^{\frac
{p}{4}}]\\
\leq &  C_{p,\delta}^{\prime}(1+\Vert\gamma_{t}\Vert^{q}))+\delta
(1+\delta)C_{p}^{4}E[(\int_{t}^{T}\Vert Z_{{\gamma_{t}}}(s)\Vert^{2}%
ds)^{\frac{p}{2}}].
\end{align*}
Choosing $\delta$ small enough to ensure $(1+\delta)\delta C_{p}^{4}<1$ thus%
\[
E[\int_{t}^{T}\Vert Z_{{\gamma_{t}}}(r)\Vert^{2}dr]\leq C_{p}(1+\Vert
\gamma_{t}\Vert^{q}),
\]
which completes the proof.
\end{proof}

We immediately have
\begin{corollary} Assuming $h$ is in $M^{2}(t,T)$, let $(Y_{{\gamma_{t}}}(s),Z_{{\gamma _{t}}}(s))_{t\leq s\leq T}$
be the adapted solution to  the  BSDE:
\begin{equation*}
Y_{{\gamma_{t}}}(s)=\Phi(B^{\gamma_{t}})+\int_{s}^{T}[f(B_{r}^{\gamma_{t}%
},Y_{{\gamma_{t}}}(r),Z_{{\gamma_{t}}}(r))+h(r)]dr-\int_{s}^{T}Z_{{\gamma_{t}}%
}(r)dB(r).
\end{equation*}
Then for any $p\geq2$, there exist $C_{p}$ and $q$ depending on
$C,T,k,p$, such that:
\begin{align}
E[\sup\limits_{s\in\lbrack t,T]}|Y_{{\gamma_{t}}}(s)|^{p}]  &  \leq
C_{p}(1+\Vert\gamma_{t}\Vert^{q}+\int^T_t|h(s)|^pds),\\
E[|\int_{t}^{T}\Vert Z_{{\gamma_{t}}}(s)\Vert^{2}ds|^{\frac{p}{2}}]
&  \leq C_{p}(1+\Vert\gamma_{t}\Vert^{q}+\int^T_t|h(s)|^pds).
\end{align}
\label{corol1}
\end{corollary}

We need the regularity properties of the solution of BSDE with respect to the
\textquotedblleft parameter\textquotedblright\ $\gamma_{t}$. For convenience,
we define $Y_{{\gamma_{t}}}(s),Z_{{\gamma_{t}}}(s)$ for any $t,s\in
\lbrack0,T],\gamma_{t}\in\Lambda$ by $Y_{{\gamma_{t}}}(s)=Y_{{\gamma_{t}}%
}(s\vee t)$, while $Z_{{\gamma_{t}}}(s)=0$ for $s<t$.

\begin{proposition}
\label{w5} For any $p\geq2$, there exist $C_{p}$ and $q$ depending on
$C,T,k,p$, such that for any $t,\bar{t}\in\lbrack0,T]$, and $\gamma_{t}%
,\bar{\gamma}_{\bar{t}}\in\Lambda$, $h,\bar{h}\in\mathbb{R}\backslash\{0\}$,%
\begin{align*}
&  \text{(i)}\,\,\,\,\,E[\sup\limits_{u\in\lbrack0,T]}|Y_{\gamma_{t}%
}(u)-Y_{\bar{\gamma}_{\bar{t}}}(u)|^{p}]\leq C_{p}(1+\Vert\gamma_{t}\Vert
^{q}+\Vert\bar{\gamma}_{\bar{t}}\Vert^{q})(|d_{\infty}(\gamma_{t},\bar{\gamma
}_{\bar{t}})|^{p}+|t-\bar{t}|^{\frac{p}{2}}),\\
&  \text{(ii)}\,\,\,\,E[|\int_{0}^{T}\Vert Z_{{\gamma_{t}}}(u)-Z_{{\bar
{\gamma}_{\bar{t}}}}(u)\Vert^{2}du|^{\frac{p}{2}}]\leq C_{p}(1+\Vert\gamma
_{t}\Vert^{q}+\Vert\bar{\gamma}_{\bar{t}}\Vert^{q})(|d_{\infty}(\gamma
_{t},\bar{\gamma}_{\bar{t}})|^{p}+|t-\bar{t}|^{\frac{p}{2}}),\\
&  \text{(iii)\thinspace\ }E[\sup\limits_{u\in\lbrack0,T]}|\Delta_{h}%
^{i}Y_{{\gamma_{t}}}(u)-\Delta_{\bar{h}}^{i}Y_{\bar{\gamma}_{\bar{t}}}%
(u)|^{p}]\\
&  \,\,\,\,\,\,\,\,\,\,\,\,\,\,\,\leq C_{p}(1+\Vert\gamma_{t}\Vert^{q}%
+\Vert\bar{\gamma}_{\bar{t}}\Vert^{q}+|h|^{q}+|\bar{h}|^{q})(|h-\bar{h}%
|^{p}+|d_{\infty}(\gamma_{t},\bar{\gamma}_{\bar{t}})|^{p}+|t-\bar{t}%
|^{\frac{p}{2}}),\\
&  \text{(iv)}\,\,E[|\int_{0}^{T}\Vert\Delta_{h}^{i}Z_{{\gamma_{t}}}%
(u)-\Delta_{\bar{h}}^{i}Z_{{\bar{\gamma}_{\bar{t}}}}(u)^{2}\Vert du|^{\frac
{p}{2}}]\\
&  \,\,\,\,\,\,\,\,\,\,\,\,\,\,\,\leq C_{p}(1+\Vert\gamma_{t}\Vert^{q}%
+\Vert\bar{\gamma}_{\bar{t}}\Vert^{q}+|h|^{q}+|\bar{h}|^{q})(|h-\bar{h}%
|^{p}+|d_{\infty}(\gamma_{t},\bar{\gamma}_{\bar{t}})|^{p}+|t-\bar{t}%
|^{\frac{p}{2}}),
\end{align*}
where $\Delta_{h}^{i}Y_{{\gamma_{t}}}(s)=\frac{1}{h}(Y_{{\gamma_{t}^{he_{i}}}%
}(s)-Y_{{\gamma_{t}}}(s))$, $\Delta_{h}^{i}Z_{{\gamma_{t}}}(s)=\frac{1}%
{h}(Z_{{\gamma_{t}^{he_{i}}}}(s)-Z_{{\gamma_{t}}}(s))$ and $(e_{1}%
,\cdots,e_{m})$ is an orthonormal basis of $\mathbb{R}^{m}$.
\end{proposition}

\begin{proof}
$(Y_{\gamma_{t}}-Y_{\bar{\gamma}_{\bar{t}}},Z_{\gamma_{t}}-Z_{\bar{\gamma
}_{\bar{t}}})$ can be formed as a linearized BSDE:
\begin{align*}
&  Y_{{\gamma_{t}}}(r)-Y_{\bar{\gamma}_{\bar{t}}}(r)\\
= &  \Phi(B^{\gamma_{t}})-\Phi(B^{\bar{\gamma}_{\bar{t}}})+\int_{r}%
^{T}[f(B_{u}^{\gamma_{t}},Y_{{\gamma_{t}}}(u),Z_{{\gamma_{t}}}(u))-f(B_{u}%
^{\bar{\gamma}_{\bar{t}}},Y_{{\bar{\gamma}_{\bar{t}}}}(u),Z_{{\bar{\gamma
}_{\bar{t}}}}(u))]du\\
&  -\int_{u}^{T}(Z_{{\gamma_{t}}}(u)-Z_{{\bar{\gamma}_{\bar{t}}}}(u))dB(u),\\
= &  \Phi(B^{\gamma_{t}})-\Phi(B^{\bar{\gamma}_{\bar{t}}})+\int_{r}^{T}%
[\alpha_{\gamma_{t},\bar{\gamma}_{\bar{t}}}(u)(Y_{{\gamma_{t}}}(u)-Y_{{\bar
{\gamma}_{\bar{t}}}}(u))+\beta_{\gamma_{t},\bar{\gamma}_{\bar{t}}%
}(u)(Z_{{\gamma_{t}}}(u)-Z_{{\bar{\gamma}_{\bar{t}}}}(u))+\hat{f}_{\gamma
_{t},\bar{\gamma}_{\bar{t}}}(u)]du\\
&  -\int_{r}^{T}(Z_{{\gamma_{t}}}(u)-Z_{{\bar{\gamma}_{\bar{t}}}%
}(u))dB(u),\ r\in\lbrack t\vee\bar{t},T],
\end{align*}
where
\begin{align*}
\alpha_{\gamma_{t},\bar{\gamma}_{\bar{t}}}(u)= &  \int_{0}^{1}\frac{\partial
f}{\partial y}(B_{u}^{\gamma_{t}},Y_{{\bar{\gamma}_{\bar{t}}}}(u)+\theta
(Y_{{\gamma_{t}}}(u)-Y_{{\bar{\gamma}_{\bar{t}}}}(u)),Z_{{\bar{\gamma}%
_{\bar{t}}}}(u)+\theta(Z_{{\gamma_{t}}}(u)-Z_{{\bar{\gamma}_{\bar{t}}}%
}(u)))d\theta,\\
\beta_{\gamma_{t},\bar{\gamma}_{\bar{t}}}(u)= &  \int_{0}^{1}\frac{\partial
f}{\partial z}(B_{u}^{\gamma_{t}},Y_{{\bar{\gamma}_{\bar{t}}}}(u)+\theta
(Y_{{\gamma_{t}}}(u)-Y_{{\bar{\gamma}_{\bar{t}}}}(u)),Z_{{\bar{\gamma}%
_{\bar{t}}}}(u)+\theta(Z_{{\gamma_{t}}}(u)-Z_{{\bar{\gamma}_{\bar{t}}}%
}(u)))d\theta,\\
\hat{f}_{\gamma_{t},\bar{\gamma}_{\bar{t}}}(u)= &  f(B_{u}^{\gamma_{t}%
},Y_{{\bar{\gamma}_{\bar{t}}}}(u),Z_{{\bar{\gamma}_{\bar{t}}}}(u))-f(B_{u}%
^{\bar{\gamma}_{\bar{t}}},Y_{{\bar{\gamma}_{\bar{t}}}}(u),Z_{{\bar{\gamma
}_{\bar{t}}}}(u)).
\end{align*}
By assumptions $(\mathbf{H1})$ and $(\mathbf{H2})$,
\[
|\hat{f}_{\gamma_{t},\bar{\gamma}_{\bar{t}}}(u)|+|\Phi(B^{\gamma_{t}}%
)-\Phi(B^{\bar{\gamma}_{\bar{t}}})|\leq2C(1+\left\Vert B_{T}^{\gamma_{t}%
}\right\Vert ^{q}+\left\Vert B_{T}^{\bar{\gamma}_{\bar{t}}}\right\Vert
^{q})\sup\limits_{u}|B^{\gamma_{t}}(u)-B^{\bar{\gamma}_{\bar{t}}}(u)|,
\]
we know that the first two inequalities hold true after applying
Corollary \ref{corol1} to the above linearized BSDE.

For the last two inequalities, we also can write $(\Delta_{h}^{i}%
Y_{{\gamma_{t}}},\Delta_{h}^{i}Z_{{\gamma_{t}}})$ as the solution of the
following linearized BSDE:
\begin{align*}
\Delta_{h}^{i}Y_{{\gamma_{t}}}(r)= &  \frac{1}{h}(\Phi(B^{\gamma_{t}^{he_{i}}%
})-\Phi(B^{\gamma_{t}}))+\int_{r}^{T}(\alpha_{\gamma_{t}^{he_{i}},\gamma_{t}%
}(u)\Delta_{h}^{i}Y_{{\gamma_{t}}}(u)\\
&  +\beta_{\gamma_{t}^{he_{i}},\gamma_{t}}(u)\Delta_{h}^{i}Z_{{\gamma_{t}}%
}(u)+\frac{1}{h}\hat{f}_{\gamma_{t}^{he_{i}},\gamma_{t}}(u))du-\int_{r}%
^{T}\Delta_{h}^{i}Z_{\gamma_{t}}(u)dB_{u}.
\end{align*}
Then the same calculus as above implies that:
\[
E[\sup\limits_{s\in\lbrack0,T]}|\Delta_{h}^{i}Y_{{\gamma_{t}}}(s)|^{p}%
+|\int_{0}^{T}\Vert\Delta_{h}^{i}Z_{{\gamma_{t}}}(s)\Vert^{2}ds|^{\frac{p}{2}%
}]\leq C_{p}(1+\Vert\gamma_{t}\Vert^{q}+|h|^{q}).
\]
Finally, we consider
\begin{align*}
\Delta_{h}^{i}Y_{{\gamma_{t}}}(r)-\Delta_{\bar{h}}^{i}Y_{\bar{\gamma}_{\bar
{t}}}(r)= &  \frac{1}{h}(\Phi(B^{\gamma_{t}^{he_{i}}})-\Phi(B^{\gamma_{t}%
}))-\frac{1}{\bar{h}}(\Phi(B^{{\bar{\gamma}}_{\bar{t}}^{\bar{h}e_{i}}}%
)-\Phi(B^{\bar{\gamma}_{\bar{t}}}))\\
&  +\int_{r}^{T}(\alpha_{\gamma_{t}^{he_{i}},\gamma_{t}}(u)\Delta_{h}%
^{i}Y_{{\gamma_{t}}}(u)-\alpha_{\bar{\gamma}_{\bar{t}}^{\bar{h}e_{i}},\bar{\gamma
}_{\bar{t}}}(u)\Delta_{\bar{h}}^{i}Y_{{\bar{\gamma}_{\bar{t}}}}(u))du\\
&  +\int_{r}^{T}(\beta_{\gamma_{t}^{he_{i}},\gamma_{t}}(u)\Delta_{h}%
^{i}Z_{{\gamma_{t}}}(u)-\beta_{\bar{\gamma}_{\bar{t}}^{\bar{h}e_{i}},\bar{\gamma
}_{\bar{t}}}(u)\Delta_{\bar{h}}^{i}Z_{{\bar{\gamma}_{\bar{t}}}}(u))du\\
&  +\int_{r}^{T}(\frac{1}{h}\hat{f}_{\gamma_{t}^{he_{i}},\gamma_{t}}%
(u)-\frac{1}{\bar{h}}\hat{f}_{\bar{\gamma}_{{\bar{t}}}^{\bar{h}e_{i}},\bar{\gamma
}_{t}}(u))du)\\
&  -\int_{r}^{T}(\Delta_{h}^{i}Z_{{\gamma_{t}}}(u)-\Delta_{\bar{h}}%
^{i}Z_{{\bar{\gamma}_{\bar{t}}}}(u))dB(u), \ r\in\lbrack
t\vee\bar{t},T]
\end{align*}
Thus $(\tilde{Y}(r),\tilde{Z}(r)):=(\Delta_{h}^{i}Y_{{\gamma_{t}}}%
(r)-\Delta_{\bar{h}}^{i}Y_{\bar{\gamma}_{\bar{t}}}(r),\Delta_{h}^{i}%
Z_{{\gamma_{t}}}(r)-\Delta_{\bar{h}}^{i}Z_{\bar{\gamma}_{\bar{t}}}(r))$ solves
the BSDE%
\begin{align*}
\tilde{Y}(r)  & =\frac{1}{h}(\Phi(B^{\gamma_{t}^{he_{i}}})-\Phi(B^{\gamma_{t}%
}))-\frac{1}{\bar{h}}(\Phi(B^{{\bar{\gamma}}_{\bar{t}}^{\bar{h}e_{i}}}%
)-\Phi(B^{\bar{\gamma}_{\bar{t}}}))\\
& +\int_{r}^{T}[\alpha_{\gamma_{t}^{he_{i}},\gamma_{t}}(u)\tilde{Y}%
(r)+\beta_{\gamma_{t}^{he_{i}},\gamma_{t}}(u)\tilde{Z}(u)+\tilde{f}%
(u)]du-\int_{r}^{T}\tilde{Z}(u)dB(u),
\end{align*}
with%
\begin{align*}
\tilde{f}(u)  & :=[\alpha_{\gamma_{t}^{he_{i}},\gamma_{t}}(u)-\alpha
_{\bar{\gamma}_{\bar{t}}^{\bar{h}e_{i}},\bar{\gamma}_{\bar{t}}}(u)]\Delta_{\bar{h}%
}^{i}Y_{{\bar{\gamma}_{\bar{t}}}}(u)+[\beta_{\gamma_{t}^{he_{i}},\gamma_{t}%
}(u)-\beta_{\bar{\gamma}_{\bar{t}}^{\bar{h}e_{i}},\bar{\gamma}_{\bar{t}}}%
(u)]\Delta_{\bar{h}}^{i}Z_{{\bar{\gamma}_{\bar{t}}}}(u)\\
& +\frac{1}{h}\hat{f}_{\gamma_{t}^{he_{i}},\gamma_{t}}(u)-\frac{1}{\bar{h}%
}\hat{f}_{\bar{\gamma}_{{\bar{t}}}^{\bar{h}e_{i}},\bar{\gamma}_{t}}(u)
\end{align*}
By assumptions $(\mathbf{H1})$ and $(\mathbf{H2})$, there exist some
$\delta,\varepsilon_{u}\in\lbrack0,1],$
\begin{align*}
&  \Phi(B^{\gamma_{t}^{he_{i}}})-\Phi(B^{\gamma_{t}})=\langle\Phi_{\gamma_{t}%
}^{\prime}(B^{\gamma_{t}^{\delta he_{i}}}),e_{i}\rangle h,\\
&  \hat{f}_{\gamma_{t}^{he_{i}},\gamma_{t}}(u)=\langle f_{\gamma_{t}}^{\prime}%
(B_{u}^{\gamma_{t}^{\varepsilon_{u}he_{i}}},Y_{\gamma_{t}}(u),Z_{\gamma_{t}%
}(u)),e_{i}\rangle h,
\end{align*}
again by the  Corollary \ref{corol1}, we know that the last two
inequalities holds true.
\end{proof}

We now establish that

\begin{proposition}
\label{w9} For each $\gamma_{t}\in\Lambda$, $\{Y_{^{\gamma^{x}_{t}}}(s),
s\in\lbrack0,T],x\in\mathbb{R}^{m}\}$ has a version which is a.e. of class
$C^{0,2}([0,T]\times\mathbb{R}^{m})$.
\end{proposition}

\begin{proof}
To simplify presentation, we shall prove only the case when $m=d=1$, as the
higher dimensional case can be treated in the same way without substantial
difficulty. Thus, for each $h,\bar{h}\in\mathbb{R}\backslash\{0\}$ and
$k,\bar{k}\in\mathbb{R}$,
\begin{align*}
&  E[\sup\limits_{s\in\lbrack0,T]}|Y_{{_{\gamma_{t}^{k}}}}(s)-Y_{{\gamma
_{t}^{\bar{k}}}}(s)|^{p}]\leq C_{p}(1+\Vert\gamma_{t}\Vert^{q})(|k-\bar
{k}|^{p}),\\
&  E[|\int_{0}^{T}|Z_{{\gamma_{t}^{k}}}(s)-Z_{{\gamma_{t}^{\bar{k}}}}%
|^{2}ds|^{\frac{p}{2}}]\leq C_{p}(1+\Vert\gamma_{t}\Vert^{q})(|k-\bar{k}%
|^{p}),\\
&  E[\sup\limits_{s\in\lbrack0,T]}|\Delta_{h}Y_{{\gamma_{t}^{k}}}%
(s)-\Delta_{\bar{h}}Y_{{\gamma_{t}^{\bar{k}}}}(s)|^{p}]\\
\leq &  C_{p}(1+\Vert\gamma_{t}\Vert^{q}+|h|^{p}+|\bar{h}|^{p})(|k-\bar
{k}|^{p}+|h-\bar{h}|^{p}),\\
&  E[|\int_{0}^{T}|\Delta_{h}Z_{{\gamma_{t}^{k}}}(s)-\Delta_{\bar{h}%
}Z_{{\gamma_{t}^{\bar{k}}}}(s)|^{2}ds|^{\frac{p}{2}}]\\
\leq &  C_{p}(1+\Vert\gamma_{t}\Vert^{q}+|h|^{p}+|\bar{h}|^{p})(|k-\bar
{k}|^{p}+|h-\bar{h}|^{p}).
\end{align*}
Therefore, using the Kolmogorov's criterion, the existence of a
continuous derivative of $Y_{{\gamma_{t}^{x}}}(s)$ with respect to
$x$ follows easily from the above estimate, as well as the existence
of a mean-square derivative of $Z_{{\gamma_{t}^{x}}}(s)$ with
respect to $x$, which is mean square
continuous in $x$. We denote them by $(D_{x}Y_{{\gamma_{t}}},D_{x}%
Z_{{\gamma_{t}}})$.

We now prove the existence of the continuous second derivative of
$Y_{{\gamma_{t}^{x}}}(s)$ with respect to $x$. By Proposition
\ref{w5}, $(D_{x}Y_{{\gamma_{t}}},D_{x}Z_{{\gamma_{t}}})$ is the
solution of the  following linearized BSDE:
\begin{align*}
D_{x}Y_{{\gamma_{t}}}(s)= &  \Phi_{\gamma_{t}}^{\prime}(B^{\gamma_{t}}%
)+\int_{s}^{T}[f_{y}(B_{r}^{\gamma_{t}},Y_{{\gamma_{t}}}(r),Z_{{\gamma_{t}}%
}(r))D_{x}Y_{{\gamma_{t}}}(r)+f_{z}(B_{r}^{\gamma_{t}},Y_{{\gamma_{t}}%
}(r),Z_{{\gamma_{t}}}(r))D_{x}Z_{{\gamma_{t}}}(r)]dr\\
&  +\ \int_{s}^{T}f_{\gamma_{t}}^{\prime}(B_{r}^{\gamma_{t}},Y_{{\gamma_{t}}%
}(r),Z_{{\gamma_{t}}}(r))dr-\ \int_{s}^{T}D_{x}Z_{{\gamma_{t}}}(r)dB(r).
\end{align*}
Then, applying Proposition \ref{w5}, we have: for each $h,\bar{h}%
\in\mathbb{R}\backslash\{0\}$ and $k,\bar{k}\in\mathbb{R}$ we have
\begin{align*}
&  E[\sup\limits_{s\in\lbrack0,T]}|\Delta_{h}D_{x}Y_{{\gamma_{t}^{k}}%
}(s)-\Delta_{\bar{h}}D_{x}Y_{{\gamma_{t}^{\bar{k}}}}(s)|^{p}]\leq
C_{p}(1+\Vert\gamma_{t}\Vert^{q})(|k-\bar{k}|^{p}+|h-\bar{h}|^{p}),\\
&  E[|\int_{0}^{T}|\Delta_{h}D_{x}Z_{{\gamma_{t}^{k}}}(s)-\Delta_{\bar{h}%
}D_{x}Z_{{\gamma_{t}^{\bar{k}}}}|^{2}ds|^{\frac{p}{2}}]\leq C_{p}%
(1+\Vert\gamma_{t}\Vert^{q})(|k-\bar{k}|^{p}+|h-\bar{h}|^{p}),
\end{align*}
which completes the proof.
\end{proof}

Now we define:
\begin{equation}
u(\gamma_{t}):=Y_{{\gamma_{t}}}(t),\ \ for\ \ \gamma_{t}\in\Lambda.\label{eq4}%
\end{equation}
By the definition of vertical derivative and Proposition $\ref{w5}$,
we have the following corollary

\begin{corollary}
\label{w1} $u(\gamma_{t})$ is $\Lambda$-continuous and
$D_{x}u(\gamma_{t})$, $D_{xx}u({\gamma_{t}})$ exist, moreover they
are both $\Lambda$-continuous. Furthermore,
$u\in\mathbb{C}_{l,lip}^{0,2}(\Lambda).$
\end{corollary}

\subsection{Path regularity of process Z.}

In Pardoux and Peng \cite{PP1}, when the terminal of BSDE is the
state-dependent case when $f=\bar{f}(t,\gamma(t),y,z)$ and $\Phi
=\varphi(\gamma(T))$, it is shown that $Z$ and $Y$ are connected in the
following sense under appropriate assumptions:%
\[
Z_{\gamma_{t}}(s)=\partial_{x}u(s,\gamma_t(t)+B(s)-B(t)).
\]
In this section, we extend this result to the path-dependent case. Indeed, we
have below a formula relating $Z$ with $Y$.\newline

\begin{proposition}
\label{w6} Under assumptions $(\mathbf{H}1)$-$(\mathbf{H}2)$, for each fixed
$\gamma_{t}\in\Lambda$, the process $(Z_{\gamma_{t}}(s))_{s\in\lbrack t,T]}$
has a continuous version with the form,
\[
D_{x}u(B_{s}^{\gamma_{t}})=Z_{{\gamma_{t}}}(s),\ \text{for each }s\in\lbrack
t,T],\ \ a.s..
\]

\end{proposition}

A direct consequence of Proposition \ref{w6} is the following corollary.

\begin{corollary}
Assume that the same conditions of Proposition \ref{w6}. Then, for the above
continuous version $Z_{{\gamma_{t}}}$, for each $p\geq2$, there exists a
constant $C_{p}>0$, depending on $C,T,k,p$, such that
\[
Z_{{\gamma_{t}}}(s)\leq C_{p}(1+\Vert B_{s}^{\gamma_{t}}\Vert^{q}),\
\ \forall s\in\lbrack t,T],\ \ P-a.s.,
\]
and
\[
E[\sup\limits_{s\in\lbrack t,T]}\Vert Z_{{\gamma_{t}}}(s)\Vert^{p}]\leq
C_{p}(1+\Vert\gamma_{t}\Vert^{q}).
\]

\end{corollary}

Before proceeding to the proof, we need the following lemma
essentially from Pardoux and Peng \cite{PP1} (the Lemma 2.5 ).

\begin{lemma}
\label{w7} Let $\gamma_{t}$ be given. For some $\bar{t}\in\lbrack
t,T]$, suppose
$\Phi(\gamma)=\varphi(\gamma(\bar{t}),\gamma(T)-\gamma(\bar{t}))$,
where $\varphi$ is in $C_{p}^{3}(\mathbb{R}^{2d},\mathbb{R}^{m})$.
Let
$f(\gamma_{s},y,z)=\bar{f}_{1}(s,\gamma_{s}(s),y,z)I_{[0,\bar{t})}(s)+\bar{f}_{2}(s,\gamma_{s}(\bar{t}),\gamma_{s}(s)-\gamma_{s}(\bar{t}),y,z)I_{[\bar
{t},T]}(s)$, where $\bar{f}_{1}:[0,T]\times\mathbb{R}^{d}\times\mathbb{R}%
^{m}\times\mathbb{R}^{m\times d}\mapsto\mathbb{R}^{m}$, $\bar{f}_{2}%
:[0,T]\times\mathbb{R}^{2d}\times\mathbb{R}^{m}\times\mathbb{R}^{m\times
d}\mapsto\mathbb{R}^{m}$ satisfy the assumption $(\mathbf{H3})$,
then for each $s\in\lbrack t,T]$,
\[
D_{x}u(B_{s}^{\gamma_{t}})=Z_{\gamma_{t}}(s),\ \ \ a.s..
\]

\end{lemma}

\begin{proof}
To simplify presentation, we shall only prove the case when $m=d=1$, as the
higher dimensional case can be treated in the same way without substantial
difficulty. \newline In this case BSDE (\ref{p1}) is rewritten, for
$s\in\lbrack t,\bar{t}]$, BSDE (\ref{p1}) is
\begin{align*}
Y_{\gamma_{s}}(u)= &
\varphi(\gamma_{s}(s)+B(\bar{t})-B(s),B(T)-B(\bar
{t}))-\int_{u}^{T}Z_{\gamma_{s}}(r)dB(r)\\
&
+\int_{u}^{T}\bar{f}_{2}(r,\gamma_{s}(s)+B(\bar{t})-B(s),B(r)-B(\bar
{t}),Y_{\gamma_{s}}(r),Z_{\gamma_{s}}(r))dr,\ \ u\in\lbrack\bar{t},T],\\
\begin{split}
Y_{\gamma_{s}}(u)= &  Y_{\gamma_{s}}(\bar{t})-\int_{u}^{\bar{t}}Z_{\gamma_{s}%
}(r)dB(r)\\
&  +\int_{s}^{\bar{t}}\bar{f}_{1}(r,\gamma(s)+B(r)-B(s),Y_{\gamma_{s}%
}(r),Z_{\gamma_{s}}(r))dr,\ \ u\in\lbrack s,\bar{t}],
\end{split}
\end{align*}
for $s\in\lbrack\bar{t},T]$, BSDE (\ref{p1}) is
\begin{align*}
Y_{\gamma_{s}}(u)= &  \varphi(\gamma_{s}(\bar{t}),\gamma_{s}(s)-\gamma
_{s}(\bar{t})+B(T)-B(s))-\int_{u}^{T}Z_{\gamma_{s}}(r)dB(r)\\
&  +\int_{u}^{T}\bar{f}_{2}(r,\gamma_{s}(\bar{t}),\gamma_{s}(s)-\gamma_{s}%
(\bar{t})+B(u)-B(s)),Y_{\gamma_{s}}(r),Z_{\gamma_{s}}(r))dr,\ \
u\in\lbrack s,T].
\end{align*}

Consider the following system of quasilinear parabolic differential equations,
defined on $(s,x,y)\in\lbrack\bar{t},T]\times\mathbb{R}^{2}$ and parameterized
by $y\in\mathbb{R}$,
\begin{align*}%
\begin{cases}
&
\partial_{s}v_{2}(s,x,y)+\frac{1}{2}\partial_{yy}v_{2}(s,x,y)+\bar{f}
_{2}(s,x,y,v_{2}(s,x,y),\partial_{y}v_{2}(s,x,y))=0,\\
& v_{2}(T,x,y)=\varphi(x,y),
\end{cases}
\end{align*}
and then, another one defined on $(s,x)\in\lbrack t,\bar{t}]\times\mathbb{R}%
$,
\begin{align*}%
\begin{cases}
& \partial_{s}v_{1}(s,x)+\frac{1}{2}\partial_{xx}v_{1}(s,x)+\bar{f}_{1}%
(s,x,v_{1}(s,x),\partial_{x}v_{1}(s,x))=0,\\
& v_{1}(\bar{t},x)=v_{2}(\bar{t},x,0).
\end{cases}
\end{align*}

Following the Theorem 3.1 and 3.2 of Pardoux and Peng \cite{PP1} and the
definition of $u$, we have $v_{2}$ is of class $C^{1,2}([\bar{t}%
,T]\times\mathbb{R}^{2},\mathbb{R})$, $v_{1}$ is of class $C^{1,2}([t,\bar
{t}]\times\mathbb{R},\mathbb{R})$, and
\[
u(\gamma_{s})=v_{1}(s,\gamma_{s}(s))I_{[t,\bar{t}]}(s)+v_{2}(s,\gamma_{s}%
(\bar{t}),\gamma_{s}(s)-\gamma_{s}(\bar{t}))I_{[\bar{t},T]}(s).
\]
Furthermore, we have, $a.s.$,
\begin{align*}
&  Y_{\gamma_{t}}(s)=v_{1}(s,\gamma_{t}(t)+B(s)-B(t)),\ \ t\leq s\leq\bar
{t},\\
&  Y_{\gamma_{t}}(s)=v_{2}(s,\gamma_{t}(t)+B(\bar{t})-B(t),B(s)-B(\bar
{t})),\ \ \bar{t}\leq s\leq T,\\
&  Z_{\gamma_{t}}(s)=\partial_{x}v_{1}(s,\gamma_{t}(t)+B(s)-B(t)),\ \ t\leq
s\leq\bar{t},\\
&  Z_{\gamma_{t}}(s)=\partial_{y}v_{2}(s,\gamma_{t}(t)+B(\bar{t}%
)-B(t),B(s)-B(\bar{t})),\ \ \bar{t}\leq s\leq T.
\end{align*}
Indeed, we can directly check it by It\^{o}'s formula and the uniqueness of BSDE.

Thus for each $s\in\lbrack t,T]$, by the definition of vertical derivative, we
have
\[
D_{x}u(B_{s}^{\gamma_{t}})=Z_{{\gamma_{t}}}(s),\ \ \ a.s..
\]
Thus $Z_{{\gamma_{t}}}(s)$ have a continuous version. In particular,%
\[
Z_{\gamma_{t}}(t)=D_{x}u(\gamma_{t}),\ \ \gamma_{t}\in\Lambda,
\]
which completes the proof.
\end{proof}

We now give the proof of Proposition \ref{w6}.\newline

\bigskip

\noindent\textbf{Proof of Proposition \ref{w6}}. {For each fixed }$t\in
\lbrack0,T)$ and{ positive integer} $n$,, we introduce a mapping $\gamma
^{n}(\bar{\gamma}_{s}):\Lambda_{s}\mapsto{\Lambda_{s}}$:
\[
\gamma^{n}(\bar{\gamma}_{s})(r)=\bar{\gamma}_{s}I_{[0,t)}(r)+\sum
\limits_{k=0}^{n-1}\gamma_{s}(t_{k+1}^{n}\wedge s)I_{[t_{k}^{n}\wedge
s,t_{k+1}^{n}\wedge s)}(r)+\gamma_{s}(s)I_{[s]}(r),
\]
where $t_{k}^{n}=t+\frac{k(T-t)}{n}$, $k=0,1,\cdots,n$, and set
\[
\Phi^{n}(\bar{\gamma}):=\Phi(\gamma^{n}(\bar{\gamma})),\ \ \ \ \ f^{n}%
(\bar{\gamma}_{s},y,z):=f(\gamma^{n}(\bar{\gamma}_{s}),y,z).
\]
For each $n$, there exist some $\varphi_{n}$ defined on $\Lambda_{t}%
\times\mathbb{R}^{n\times d}$ and $\psi_{n}$ defined on $[t,T]\times
\Lambda_{t}\times\mathbb{R}^{n\times d}\times\mathbb{R}^{m}\times
\mathbb{R}^{m\times d}$ such that,
\begin{align*}
&  \Phi^{n}(\bar{\gamma})=\varphi_{n}(\bar{\gamma}_{t},\bar{\gamma}(t_{1}%
^{n})-\bar{\gamma}(t),\cdots,\bar{\gamma}(t_{n}^{n})-\bar{\gamma}(t_{n-1}%
^{n})),\\
&  f^{n}(\bar{\gamma}_{s},y,z)=\psi_{n}(s,\bar{\gamma}_{t},\bar{\gamma}%
_{s}(t_{1}^{n}\wedge s)-\bar{\gamma}_{s}(t),\cdots,\gamma_{s}(t_{n}^{n}\wedge
s)-\gamma_{s}(t_{n-1}^{n}\wedge s),y,z).
\end{align*}
Indeed, set
\begin{align*}
&  \bar{\varphi}_{n}(\bar{\gamma}_{t},x_{1},\cdots,x_{n}):=\Phi(\bar{\gamma
}(s)I_{[0,t)}(s)+\sum\limits_{k=1}^{n}x_{i}I_{[t_{k-1}^{n},t_{k}^{n}%
)}(s)+x_{n}I_{[T]}(s)),\\
&  \varphi_{n}(\bar{\gamma}_{t},x_{1},\cdots,x_{n}):=\bar{\varphi}_{n}%
(\bar{\gamma}_{t},\bar{\gamma}(t)+x_{1},\bar{\gamma}(t)+x_{1}+x_{2}%
,\cdots,\bar{\gamma}(t)+\sum\limits_{i=1}^{n}x_{i}).
\end{align*}
By the assumptions $(\mathbf{H}1)$ and $(\mathbf{H}2)$, $\varphi_{n}%
(\bar{\gamma}_{t}^{x_{0}},x_{1},\cdots,x_{n})$ is a $C_{p}^{3}$-function of
$x_{0}$,$\cdots,x_{n}$ for each fixed $\bar{\gamma}_{t}$, in particular,
$\partial_{x_{i}}\varphi_{n}(x_{1},\cdots,x_{n})=\Phi_{\gamma_{t_{i-1}^{n}}%
}^{\prime}(\bar{\gamma}(s)I_{[0,t)}(s)+\sum\limits_{k=1}^{n}x_{i}%
I_{[t_{k-1}^{n},t_{k}^{n})}(s)+x_{n}I_{[T]}(s))$. Furthermore, we
can check that
$\psi_{n}(s,\bar{\gamma}_{t}^{x_{0}},x_{1},\cdots,x_{n},y,z)$
satisfy assumption $(\mathbf{H3})$ for each fixed
$\bar{\gamma}_{t}$.

Consider the following BSDE, for any $\bar{t}\geq t$, $\bar{\gamma}_{\bar{t}%
}\in\Lambda_{\bar{t}}$,
\[
Y_{\bar{\gamma}_{\bar{t}}}^{(n)}(s)=\Phi^{(n)}(B^{\bar{\gamma}_{\bar{t}}%
})+\int_{s}^{T}f^{n}(B_{r}^{\bar{\gamma}_{\bar{t}}},Y_{\bar{\gamma}_{\bar{t}}%
}^{(n)}(r),Z_{\bar{\gamma}_{\bar{t}}}^{(n)}(r))dr-\int_{s}^{T}Z_{\bar{\gamma
}_{\bar{t}}}^{(n)}(r)dB(r),\ \ s\in\lbrack\bar{t},T],
\]
and set
\[
u^{(n)}(\bar{\gamma}_{\bar{t}})=Y_{\bar{\gamma}_{\bar{t}}}^{(n)}%
(\bar{t}),\ \ \ \ \bar{\gamma}_{\bar{t}}\in\Lambda.
\]
Iterating the argument as the above Lemma \ref{w7}, we can get for each
$s\in\lbrack t,T]$,
\[
D_{x}u^{(n)}(B_{s}^{\gamma_{t}})=Z_{\gamma_{t}}^{(n)}(s),\ \ a.s..
\]
By the Corollary \ref{corol1},
\begin{align*}
&  |u^{(n)}(\bar{\gamma}_{\bar{t}})-u(\bar{\gamma}_{\bar{t}})|\\
\leq &  CE[|\Phi^{(n)}(B^{\bar{\gamma}_{\bar{t}}})-\Phi(B^{\bar{\gamma}%
_{\bar{t}}})|^{2}+\int_{0}^{T}|f(B_{r}^{\bar{\gamma}_{\bar{t}}},Y_{\bar
{\gamma}_{\bar{t}}}^{(n)}(r),Z_{\bar{\gamma}_{\bar{t}}}^{(n)}(r))-f^{n}%
(B_{r}^{\bar{\gamma}_{\bar{t}}},Y_{\bar{\gamma}_{\bar{t}}}^{(n)}%
(r),Z_{\bar{\gamma}_{\bar{t}}}^{(n)}(r))|^{2}dr]^{\frac{1}{2}}\\
\leq &  C_{1}E[|(1+\Vert\bar{\gamma}_{\bar{t}}\Vert^{k}+\sup\limits_{s}%
|B(s)-B(t)|^{k})\Vert\gamma^{n}(B^{\bar{\gamma}_{\bar{t}}})-B^{\bar{\gamma
}_{\bar{t}}}\Vert^{2}]^{\frac{1}{2}}\\
\leq &  C_{1}((1+\Vert\bar{\gamma}_{\bar{t}}\Vert^{k})(E[\sup\limits_{s}%
|B(s)-\sum\limits_{k=1}^{n-1}B(t_{k+1}^{n})I_{[t_{k}^{n},t_{k+1}^{n})}%
(s)|^{4}]+\Vert\gamma^{n}(\bar{\gamma}_{\bar{t}})-\bar{\gamma
}_{\bar{t}}\Vert^4)^{\frac{1}{4}}\\
\leq &  C_{2}(1+\Vert\bar{\gamma}_{\bar{t}}\Vert^{k})(\frac{1}{n^4}+\Vert\gamma^{n}(\bar{\gamma}_{\bar{t}})-\bar{\gamma
}_{\bar{t}}\Vert).
\end{align*}
Moreover, we have
\begin{align*}
&  |D_{x}u^{(n)}(\bar{\gamma}_{\bar{t}})-D_{x}u(\bar{\gamma}_{\bar{t}})|\leq
 C_{2}(1+\Vert\bar{\gamma}_{\bar{t}}\Vert^{k})(\frac{1}{n^4}+\Vert\gamma^{n}(\bar{\gamma}_{\bar{t}})-\bar{\gamma
}_{\bar{t}}\Vert),,\\
&  |D_{xx}u^{(n)}(\bar{\gamma}_{\bar{t}})-D_{xx}u(\bar{\gamma}_{\bar{t}})|\leq
 C_{2}(1+\Vert\bar{\gamma}_{\bar{t}}\Vert^{k})(\frac{1}{n^4}+\Vert\gamma^{n}(\bar{\gamma}_{\bar{t}})-\bar{\gamma
}_{\bar{t}}\Vert).
\end{align*}
It follows from, for each $p\geq2$,
\begin{align*}
&\lim\limits_{n}E[\sup\limits_{s\in\lbrack t,T]}|D_{x}u^{(n)}(B_{s}^{\gamma
_{t}})-D_{x}u(B_{s}^{\gamma_{t}})|^{p}]\\
\leq & C_2 \lim\limits_{n}E[\sup\limits_{s\in\lbrack t,T]}|(1+\Vert B_{s}^{\gamma_{t}} \Vert^{k})(\frac{1}{n^4}+\Vert\gamma^{n}(B_{s}^{\gamma_{t}})-B_{s}^{\gamma_{t}}\Vert)|^{p}]\\
=&0,
\end{align*}
and $\lim\limits_{n}E[\int_{t}^{T}|Z_{\gamma_{t}}(u)-Z_{\gamma_{t}}%
^{(n)}(u)|^{2}du|^{\frac{p}{2}}]=0$ that
\[
D_{x}u(B_{s}^{\gamma_{t}})=Z_{\gamma_{t}}(s),\ \ dP\times ds-a.e.\ \text{on}%
\ [t,T],
\]
which completes the proof.

\section{Path-dependent PDE}

We now relate our BSDE to the following system of path-dependent version of
the Kolmogorov backward equation:
\begin{align}%
\begin{cases}
\label{w8} & D_{t}u(\gamma_{t})+\frac{1}{2}D_{xx}u(\gamma_{t})+f(\gamma
_{t},u(\gamma_{t}),D_{x}u(\gamma_{t}))=0,\ \gamma_{t}\in\Lambda,\ t\in
\lbrack0,T),\\
& u(\gamma)=\Phi(\gamma),\ \ \gamma\in\Lambda_{T.}%
\end{cases}
\end{align}
where $u:\Lambda\mapsto\mathbb{R}^{m}$ is a function on $\Lambda$.
We immediately obtain.

\begin{theorem}
\label{Th4.1}Assume that assumptions $(\mathbf{H1})$ and $(\mathbf{H2})$ hold
and let $u\in\mathbb{C}^{1,2}(\Lambda)$ be a solution of the equation
(\ref{w8}). Then we have $u(\gamma_{t})=Y_{{\gamma_{t}}}(t)$, for each
$\gamma_{t}\in{\Lambda}$, where $(Y_{{\gamma_{t}}}(s),Z_{{\gamma_{t}}%
}(s))_{t\leq s\leq T}$ is the unique solution of the BSDE (\ref{p1}).
Consequently, the path-dependent PDE (\ref{w8}) has at most one solution.
\end{theorem}

\begin{proof}
Applying the functional It\^{o} formula (\ref{w2}) to $u(B_{s}^{\gamma_{t}})$
on $s\in\lbrack t,T)$, we have%
\[
du(B_{s}^{\gamma_{t}})=(D_{s}u(B_{s}^{\gamma_{t}})+\frac{1}{2}D_{xx}%
u(B_{s}^{\gamma_{t}}))ds+D_{x}u(B_{s}^{\gamma_{t}})dB(s)\text{.}%
\]
Since $u$ solves PDE (\ref{w8}), thus%
\[
-du(B_{s}^{\gamma_{t}})=f(B_{s}^{\gamma_{t}},u(B_{s}^{\gamma_{t}}%
),D_{x}u(B_{s}^{\gamma_{t}}))ds-D_{x}u(B_{s}^{\gamma_{t}})dB(s)\text{,}%
\]
which, with $u(B_{T}^{\gamma_{t}})=\Phi(B_{T}^{\gamma_{t}})$ and $u\in
C^{1,2}(\Lambda)$, implies that $(Y_{{\gamma_{t}}}(s),Z_{{\gamma_{t}}%
}(s))=(u(B^{\gamma_{t}}(s)),D_{x}u(B_{s}^{\gamma_{t}}(s))$ is the unique
solution of BSDE (\ref{p1}). In particular $u(\gamma_{t})=Y_{{\gamma_{t}}}%
(t)$, which completes the proof.
\end{proof}

By using this theorem and the classical comparison theorem of BSDE
(Lemma \ref{comp}), we have the comparison theorem of Path-dependent
PDE:

\begin{corollary}
We assume $m=1$ and that $f=f_{i}$, $\Phi=\Phi_{i}$, $i=1,2$ satisfy
the same assumptions as in Theorem \ref{Th4.1} as well as:\newline$\
\bullet$ $f_{1}(\gamma_{t},y,z)\leq f_{2}(\gamma_{t},y,z)$, for each
$(\gamma _{t},y,z)\in(\Lambda\times\mathbb{R}\times\mathbb{R}^{d})$;
\newline$\bullet$ $\Phi_{1}(\gamma_{T})\leq\Phi_{2}(\gamma_{T})$,
for each $\gamma_{T}\in \Lambda_{T}$. \newline Let
$u_{i}\in\mathbb{C}^{1,2}(\Lambda)$ be the solution of equation
(\ref{w8}) associated with $(f,\Phi)=(f_{i},\Phi_{i})$, $i=1,2$.
Then we also have $u_{1}(\gamma_{t})\leq u_{2}(\gamma_{t})$, for
each $\gamma_{t}\in\Lambda$.
\end{corollary}

We are now in a position to prove the converse to the above result :

\begin{theorem}
\label{w10} We make assumptions $(\mathbf{H}1)$-$(\mathbf{H}2)$. Then the
function $u$ defined in (\ref{eq4}) is the unique $\mathbb{C}^{1,2}(\Lambda
)$-solution of the path-dependent PDE (\ref{w8}).
\end{theorem}

\begin{proof}
From Corollary \ref{w1}, $u\in\mathbb{C}^{0,2}(\Lambda)$. Let $\delta\geq0$ be
such that $t+\delta\leq T$. From the definition of $u$,
\[
u(B_{t+\delta}^{\gamma_{t}})=Y_{{\gamma_{t}}}({t+\delta}).
\]

Hence
\[
u(\gamma_{t,t+\delta})-u(\gamma_{t})=u(\gamma_{t,\delta})-u(B_{t+\delta}%
^{\gamma_{t}})+u(B_{t+\delta}^{\gamma_{t}})-u(\gamma_{t}).
\]
By the proof of Proposition \ref{w6}, we get
\begin{align*}
u(\gamma_{t,t+\delta})-u(\gamma_{t})= &  \lim\limits_{n\rightarrow\infty
}[u^{(n)}(\gamma_{t,t+\delta})-u^{(n)}(B_{t+\delta}^{\gamma_{t}})]\\
&  -\int_{t}^{t+\delta}f(B_{s}^{\gamma_{t}},Y_{{\gamma_{t}}}(s),Z_{{\gamma
_{t}}}(s))ds+\int_{t}^{t+\delta}Z_{{\gamma_{t}}}(s)dB(s).
\end{align*}
From Lemma \ref{w7} and the Proposition 3.2 of Pardoux and Peng \cite{PP1}, we
have
\begin{align*}
&  \ \ \ \ u^{(n)}(\gamma_{t,t+\delta})-u^{(n)}(B_{t+\delta}^{\gamma_{t}})\\
= &  \int_{t}^{t+\delta}D_{s}u^{(n)}(\gamma_{t,s})ds-\int_{t}^{t+\delta
}D_{s}u^{(n)}(B_{s}^{\gamma_{t}})ds\\
&  -\int_{t}^{t+\delta}D_{x}u^{(n)}(B_{s}^{\gamma_{t}})dB(s)-\frac{1}{2}%
\int_{t}^{t+\delta}D_{xx}u^{(n)}(B_{s}^{\gamma_{t}})ds.
\end{align*}
Thus
\begin{align*}
&  u(\gamma_{t,t+\delta})-u(\gamma_{t})\\
= &  -\int_{t}^{t+\delta}D_{x}u(B_{s}^{\gamma_{t}})dB(s)-\frac{1}{2}\int%
_{t}^{t+\delta}D_{xx}u(B_{s}^{\gamma_{t}})ds\\
&  -\int_{t}^{t+\delta}f(B_{s}^{\gamma_{t}},Y_{{\gamma_{t}}}(s),Z_{{\gamma
_{t}}}(s))ds+\int_{t}^{t+\delta}Z_{{\gamma_{t}}}(s)dB(s)+\lim
\limits_{n\rightarrow\infty}C^{n},
\end{align*}
where
\[
C^{n}=\int_{t}^{t+\delta}D_{s}u^{n}(\gamma_{t,s})ds-\int_{t}^{t+\delta}%
D_{s}u^{n}(B_{s}^{\gamma_{t}})ds.
\]
It is easy to check that:%
\[
|C^{n}|\leq C\delta\sup\limits_{s\in\lbrack t,t+\delta]}|B(s)-B(t)|.
\]
Taking expectation on both sides and following the Proposition
\ref{w6},
\[
\lim\limits_{\delta\rightarrow0}\frac{u(\gamma_{t,t+\delta})-u(\gamma_{t}%
)}{\delta}=-\frac{1}{2}D_{xx}u(\gamma_{t})-f(\gamma_{t},u(\gamma_{t}%
),D_{x}u(\gamma_{t})),
\]
hence, $u\in\mathbb{C}^{1,2}(\Lambda)$ and it satisfies the equation
$(\ref{w8})$.
\end{proof}
\begin{corollary}
 We make assumptions $(\mathbf{H}1)$-$(\mathbf{H}2)$. Then $(u(B_t),D_xu(B_t))$ is the unique solution of the BSDE (\ref{BSDE1}).
\end{corollary}

\begin{remark}
In the case where $\Phi(\gamma)=\varphi(\gamma(T))$, for some $\varphi\in
C_{p}^{3}(\mathbb{R}^{m})$ and $f$ satisfies assumption $(\mathbf{H3})$, the
above results is the nonlinear Feynman-Kac formula, which given by Peng
\cite{Peng1} and Pardoux-Peng \cite{PP2}.
\end{remark}

\begin{remark}
Suppose $k=1$, and let
\[
f(t,y,z)=c(t)y,
\]
where $f$ satisfy assumption $(\mathbf{H3})$.  In that case the BSDE (2) has
the explicit solution :
\[
Y_{{\gamma_{t}}}(s)=\Phi(B^{\gamma_{t}})e^{\int_{s}^{T}c(r)dr}-\int_{s}%
^{T}e^{\int_{s}^{u}c(r)dr}Z_{{\gamma_{t}}}(u)dB(u),
\]
and
\[
Y_{{\gamma_{t}}}(t)=E[\Phi(B^{\gamma_{t}})e^{\int_{t}^{T}c(r)dr}].
\]

\end{remark}

\begin{example}
Suppose $\Phi:\Lambda_{T}\mapsto\mathbb{R}$ :
\[
\Phi(\gamma)=\int_{0}^{T}\varphi(\gamma(s))ds,
\]
for some $\varphi\in C_{b}^{3}(\mathbb{R})$. It is obvious that $\Phi$ satisfies assumption $(\mathbf{H1})$.

From the above remark, for each $t\in\lbrack0,T],\gamma_{t}\in\Lambda_{t}$
\[
u(\gamma_{t})=\int_{0}^{t}\varphi(\gamma_{t}(s))dse^{\int_{t}^{T}c(r)dr}%
+\int_{t}^{T}e^{\int_{t}^{T}c(r)dr}E[\varphi(\gamma_{t}(t)+B(s)-B(t))]ds.
\]
By the classic Feyman-Kac formula, $\forall s\in\lbrack0,T]$ and
$x\in\mathbb{R}$,
\[
u^{s}(t,x)=E[\varphi(x+B(s)-B(t))],\ \ \ t\leq s
\]
is the solution of the following parabolic differential equation :
\[%
\begin{cases}
\frac{\partial u^{s}}{\partial t}+\frac{1}{2}\frac{\partial^{2}u^{s}}{\partial
x^{2}}=0,\ \ \ t\in\lbrack0,s)\\
u^{s}(s,x)=\varphi(x).
\end{cases}
\]
then
\[
u(\gamma_{t})=\int_{0}^{t}\varphi(\gamma_{t}(s))dse^{\int_{t}^{T}c(r)dr}%
+\int_{t}^{T}e^{\int_{t}^{T}c(r)dr}u^{s}(t,\gamma_{t}(t))ds.
\]
By the definitions of horizontal derivative and vertical derivative, thus
\begin{align*}
&  D_{t}u(\gamma_{t})=-c(t)u(\gamma_{t})+e^{\int_{t}^{T}c(r)dr}\int_{t}%
^{T}\partial_{t}u^{s}(t,\gamma_{t}(t))ds,\\
&  D_{x}u(\gamma_{t})=e^{\int_{t}^{T}c(r)dr}\int_{t}^{T}\partial_{x}%
u^{s}(t,\gamma_{t}(t))ds,\\
&  D_{xx}u(\gamma_{t})=e^{\int_{t}^{T}c(r)dr}\int_{t}^{T}\partial_{xx}%
^{2}u^{s}(t,\gamma_{t}(t))ds.
\end{align*}
It is obvious that
\[
D_{t}u(\gamma_{t})+\frac{1}{2}D_{xx}u(\gamma_{t})=-c(t)u(\gamma_{t}),
\]
which satisfies the equation $(\ref{w8})$.
\end{example}


                                                                                          %

\end{document}